\documentclass[11pt]{amsart}

\usepackage{amsmath,amssymb,latexsym,amsthm,newlfont,enumerate}
\usepackage{thmtools,enumitem,csquotes}
\usepackage{etoolbox,changepage}
\usepackage{hyperref}
\usepackage{color} 
\usepackage{tikz-cd}

\theoremstyle{definition}

\theoremstyle{plain}

\newtheorem{thm}{Theorem}[section]

\newtheorem{lem}[thm]{Lemma}

\newtheorem{cor}[thm]{Corollary}

\newtheorem{thmA}{Theorem}

\newtheorem{corA}[thmA]{Corollary}

\theoremstyle{definition}

\newtheorem{dfn}[thm]{Definition}

\newtheorem{rem}[thm]{Remark}

\theoremstyle{remark}

\newcommand{\Z}{\mathbb{Z}}
\newcommand{\N}{\mathbb{N}}
\newcommand{\C}{\mathbb{C}}
\newcommand{\R}{\mathbb{R}}
\newcommand{\Q}{\mathbb{Q}}

\newcommand{\OO}{\mathcal{O}}

\DeclareMathOperator{\mult}{mult}

\DeclareMathOperator{\Supp}{Supp}

\DeclareMathOperator{\Div}{Div}

\DeclareMathOperator{\NEb}{\overline{\mathrm{NE}}}

\DeclareMathOperator{\Mov}{Mov}

\begin{document}
	\title[On the existence of minimal models]{On the existence of minimal models\\ for log canonical pairs}
	
	\author{Vladimir Lazi\'c}
	\address{Fachrichtung Mathematik, Campus, Geb\"aude E2.4, Universit\"at des Saarlandes, 66123 Saarbr\"ucken, Germany}
	\email{lazic@math.uni-sb.de}
	
	\author{Nikolaos Tsakanikas}
	\address{Fachrichtung Mathematik, Campus, Geb\"aude E2.4, Universit\"at des Saarlandes, 66123 Saarbr\"ucken, Germany}
	\email{tsakanikas@math.uni-sb.de}
	
	\thanks{
		We were supported by the DFG-Emmy-Noether-Nachwuchsgruppe ``Gute Strukturen in der h\"oherdimensionalen birationalen Geometrie". We would like to thank O.\ Fujino for pointing out a gap in a previous proof of Theorem \ref{thm:HH_gMM} and for useful comments, S.\ Filipazzi, Y.\ Gongyo, J.\ Han, W.\ Liu, F.\ Meng and J.\ Moraga for useful discussions related to this work and the referee for valuable comments.
	}
	
\begin{abstract}
We show that minimal models of log canonical pairs exist, assuming the existence of minimal models of smooth varieties. 
\end{abstract}
	
	\maketitle
	\setcounter{tocdepth}{1}
	\tableofcontents
	
\section{Introduction}

The goal of this paper is to reduce the problem of the existence of minimal models for log canonical pairs to the problem of the existence of minimal models for smooth varieties. The following is our main result.

\begin{thmA} \label{thm:mainthm}
The existence of minimal models for smooth varieties of dimension $n$ implies the existence of minimal models for log canonical pairs of dimension $n$.
\end{thmA}

Note that \emph{minimal models} in the above theorem are \emph{relative} minimal mo\-dels, that is, minimal models of quasi-projective varieties or pairs which are projective and whose canonical class is pseudoeffective over another normal quasi-projective variety. In this paper, minimal models are meant in the usual sense and not in the sense of Birkar-Shokurov; for the differences, see \S\ref{sec:MM}.

In order to state various results of this paper, we need the notions of NQC g-pairs and NQC weak Zariski decompositions; for their definitions, see Section \ref{sec:prelim}. In all these results, the assumption in lower dimensions means the existence of \emph{relative} minimal models as described above. For simplicity and clarity, below we state results mostly for usual pairs. Their counterparts for NQC g-pairs are stated and proved in Section \ref{sec:main1}.

The next result improves both the assumptions in lower dimensions and the conclusions of \cite[Corollary 1.7]{Bir11} and \cite[Theorem 1.5]{Bir12b}, and leads to a similar refinement of \cite[Corollary 1.6]{Bir12b}. 

\begin{thmA}\label{thm:maincorollary}
	Assume the existence of minimal models for smooth varieties of dimension $n-1$. 
	
	Let $ (X/Z,\Delta) $ be a log canonical pair of dimension $ n $. The following are equivalent:
	\begin{enumerate}[label=(\roman*)]
		\item $ (X,\Delta) $ has an NQC weak Zariski decomposition over $Z$,
		\item $ (X,\Delta) $ has a minimal model over $Z$. 
	\end{enumerate} 
\end{thmA}

Our results are stronger when a general fibre of the structure morphism has a lot of rational curves:

\begin{thmA} \label{thm:uniruled}
	Assume the existence of minimal models for smooth varieties of dimension $n-1$. 
	
	Let $ (X/Z,\Delta) $ be a pseudoeffective log canonical pair of dimension $ n $ such that a general fibre of the morphism $ X\to Z $ is uniruled. Then $ (X,\Delta) $ has a minimal model over $Z$. 
\end{thmA} 

We immediately derive the following corollary. Part (iii) is new. Part (i) was originally proved in \cite{Sho09} and another proof was given later in \cite{Bir12b}, while part (ii) was first proved in \cite{Bir10}. However, we stress that, unlike the results in these references, we prove here the existence of minimal models in the usual sense.

\begin{corA}\label{cor:dim4+5}
	Let $(X/Z, \Delta)$ be a pseudoeffective log canonical pair.
	\begin{enumerate}
		\item[(i)] If $\dim X=4$, then $(X, \Delta)$ has a minimal model over $Z$. 
		\item[(ii)] If $\dim X=5$, then $(X, \Delta)$ has a minimal model over $Z$ if and only if it admits an NQC weak Zariski decomposition over $Z$. In particular, if $K_X+\Delta$ is effective over $Z$, then it has a minimal model over $Z$. 
		\item[(iii)] If $ \dim X = 5 $ and a general fibre of the morphism $ X\to Z $ is uniruled, then 
		$ (X,\Delta) $ has a minimal model over $ Z $.
	\end{enumerate}
\end{corA}

If in the previous result the underlying variety $ X $ is a rationally connected $5$-fold and the pair $(X,\Delta)$ is klt, then the conclusion follows by \cite[Theorem 4.1]{Gon15} and by Corollary \ref{cor:dim4+5}(ii); this was discussed with Y.\ Gongyo.

\medskip

In order to prove Theorem \ref{thm:mainthm}, we use strategies from the very recent papers \cite{HanLi} and \cite{HM18}, building to a large extent on previous works of Birkar. In brief, we may assume that $X$ is smooth and we consider divisors $K_X+(1-\varepsilon)\Delta$ which are pseudoeffective. Then there are two cases. If one can take $\varepsilon=1$, then by assumption one obtains a weak Zariski decomposition of $K_X$, which in turn gives a weak Zariski decomposition of $K_X+\Delta$. Otherwise, we consider the maximal number $\varepsilon<1$ such that $K_X+(1-\varepsilon)\Delta$ is pseudoeffective, we run a suitable MMP in order to show that this divisor has a weak Zariski decomposition by lifting one from a lower-dimensional variety, and thus obtain a weak Zariski decomposition of $K_X+\Delta$. Finally, in both cases we use the main result of \cite{HanLi} together with a significant input from \cite{HH19} to conclude, modulo a suitable statement for generalised pairs in lower dimensions. 

One then has to improve the inductive assumption in lower dimensions in order to be able to work with usual pairs and not with generalised pairs. To a large extent, this was done in \cite{HM18}. Our main technical result, which is proved in Section \ref{sec:main} and improves considerably on \cite[Theorem 2]{HM18}, is:

\begin{thmA} \label{thm:HMref}
	The existence of NQC weak Zariski decompositions for smooth varieties of dimension $n$ implies the existence of NQC weak Zariski decompositions for NQC log canonical g-pairs of dimension $n$.
\end{thmA}

\medskip

Now, we recall approaches towards the existence of minimal models of log canonical pairs which have led to the proof of the main result of this paper. The paper \cite{Bir10} shows the existence of minimal models for \emph{effective} log canonical pairs, modulo the termination of flips with scaling in lower dimensions. For such a pair $(X,\Delta)$, one uses an effective divisor numerically equivalent to $K_X+\Delta$ to define a quantity which behaves in a controllable way when one tries to construct minimal models. The main issue with this statement is that it assumes a strong conjecture -- the Nonvanishing conjecture -- in order to derive a statement about the existence of minimal models. It is, however, expected that in order to prove the Nonvanishing conjecture, one first has to pass to a minimal model, see \cite{LP18a,LP20b}.

It was realised later in \cite{Bir12b} that the arguments work very well if instead one assumes something much weaker, namely that one has a so-called weak Zariski decomposition. Roughly speaking, instead of working with divisors which are effective, one works with sums of an effective and a nef divisor. This offers enough flexibility in order to run induction better. 

Another issue is the very often overlooked assumption in lower dimensions. For instance, in \cite{Bir12b} one assumes the termination of flips in lower dimensions; in particular, one needs the termination to hold not only for pseudoeffective pairs but also for non-pseudoeffective pairs. The problem is that we currently do not have an inductive statement which works for non-pseudoeffective pairs.

A very satisfying solution, dealing with both of the above issues, was given in \cite{HanLi}, albeit by working in the larger category of generalised pairs; a sketch of a strategy was outlined in \cite[\S6]{BH14}. Generalised pairs were introduced in \cite{BH14,BZ16} and have emerged as key players in different contexts ranging from the BAB Conjecture \cite{Bir16b} and the singular version of Fujita's spectrum conjecture \cite{HanLi17} to problems around the termination of flips \cite{Mor18} and the Generalised Nonvanishing and Abundance conjectures \cite{LP20a,LP20b,HanLiu}. 

\medskip

As a side product, we obtain a few results related to the termination of flips. The first one deals with the termination of flips with scaling of an ample divisor; it complements \cite[Theorem 1.9]{Bir12a} and \cite[Corollary 2.9]{HX13}, and is based on \cite[Theorem 1.7]{HanLi} and \cite[Theorem 1.7]{HH19}.

\begin{thmA}\label{thm:scaling}
	Assume the existence of minimal models for smooth varieties of dimension $n-1$.
	
	Let $ (X/Z,\Delta) $ be a log canonical pair of dimension $ n $. If $ (X,\Delta) $ has an NQC weak Zariski decomposition over $Z$, then there exists a 
	$(K_X + \Delta)$-MMP with scaling of an ample divisor over $Z$ which terminates. In particular, the pair 
	$ (X,\Delta) $ has a minimal model over $ Z $. 
\end{thmA}

Finally, combining \cite[Theorem 1]{HM18} with Theorem \ref{thm:mainthm} we obtain: 

\begin{corA} \label{cor:secondary}
	Assume the termination of flips for NQC klt g-pairs of dimensions at most $ n-1 $. Then the existence of minimal models for smooth varieties of dimension $ n $ implies the termination of flips for pseudoeffective NQC log canonical g-pairs of dimension $ n $.
\end{corA}

As mentioned above, the main problem with this statement is the assumption in lower dimensions, since we assume the termination of flips even for
non-pseudoeffective pairs. We hope to address this issue in a future work.

\section{Preliminaries}\label{sec:prelim}

Throughout the paper we work over $\C$ and all varieties are normal and quasi-projective. A variety $X$ over another variety $Z$, denoted by $X/Z$, is always assumed to be projective over $Z$. 

Given a variety $ X$ over $ Z $ and two $ \R $-Cartier divisors $ D_1 $ and $ D_2 $ on $ X $, we say that $ D_1 $ and $ D_2 $ are \emph{numerically equivalent over $ Z $}, denoted by $ D_1 \equiv_Z D_2 $, if $ D_1\cdot C = D_2 \cdot C $ for any curve $ C $ contained in a fibre of the morphism $ X \to Z $.

Given a normal projective variety $ X $ and a pseudoeffective $\R$-Cartier $\R$-divisor $ D $ on $X$, we denote by $ \nu(X,D) $ the \emph{numerical dimension} of $ D $, see \cite[Chapter V]{Nak04}, \cite{Kaw85}.

Given a smooth projective variety $ X $ and a pseudoeffective $\R$-divisor $ D $ on $X$, we denote by $P_\sigma(D)$ and $N_\sigma(D)$ the $\R$-divisors forming the \emph{Nakayama-Zariski decomposition} of $D$, see\ \cite[Chapter III]{Nak04}. 

A \emph{fibration} is a projective surjective morphism with connected fibres. A \emph{birational contraction} is a birational map whose inverse does not contract any divisors.

\begin{dfn}
	Let $X$ and $Y$ be normal varieties and let $\varphi \colon X\dashrightarrow Y$ be a birational contraction. Let $D$ be an $\R$-Cartier $\R$-divisor on $X$ and assume that $\varphi_*D$ is $\R$-Cartier. Then $\varphi$ is \emph{$D$-non-positive} (respectively \emph{$D$-negative}) if there exists a smooth resolution of indeterminacies $(p,q)\colon W\to X\times Y$ of $\varphi$ such that 
	$$ p^*D\sim_\R q^* \varphi_*D + E,$$
	where $E$ is an effective $q$-exceptional $\R$-Cartier divisor on $W$ (respectively $E$ is an effective  $q$-exceptional $\R$-Cartier divisor on $W$ whose support contains the strict transform of every $\varphi$-exceptional divisor). 
\end{dfn} 

The following results are well-known.

\begin{lem}\label{lem:numtriv}
	Let $f\colon X\to Y$ be a projective surjective morphism between normal quasi-projective varieties. Then there exists an open subset $U\subseteq Y$ with the following property: if $D$ is an $\R$-Cartier $\R$-divisor on $X$ which is numerically trivial on a fibre of $ f $ over some point of $U$, then $D|_{f^{-1}(U)}\equiv_U0$.
\end{lem}

\begin{proof}
	Let $\pi\colon\widehat X\to X$ be a desingularisation of $ X $ and let $U$ be an open subset of $Y$ such that the morphism $f\circ \pi$ is smooth over $U$. Let $D$ be an $\R$-Cartier $\R$-divisor on $X$ such that $D|_{f^{-1}(s)}\equiv0$ for some $s\in U$. Then $(\pi^*D)|_{(f \circ \pi)^{-1}(s)}\equiv0$. It follows by \cite[Lemma II.5.15(3)]{Nak04} that $(\pi^*D)|_{(f\circ\pi)^{-1}(U)}\equiv_U 0$. Consequently $D|_{f^{-1}(U)}\equiv_U0$. 
\end{proof}

\begin{lem}\label{lem:Grauert}
Let $f\colon X\to Y$ be a fibration between normal quasi-projective varieties. Let $D$ be a $\Q$-Cartier $\Q$-divisor on $X$.
\begin{enumerate}
\item[(i)] If $\kappa(F,D|_F)\geq0$ for a very general fibre $F$ of $f$, then there exists an effective $\Q$-divisor $E$ on $X$ such that $D\sim_{\Q,Y}E$.
\item[(ii)] If $D|_F\sim_\Q0$ for a very general fibre $F$ of $f$, then there exists a non-empty open subset $U\subseteq Y$ such that $D|_{f^{-1}(U)}\sim_{\Q,U}0$. 
\end{enumerate}
\end{lem}

\begin{proof}
We first show (i). We may assume that $D$ is Cartier. By assumption there exists a subset $V\subseteq Y$, which is the intersection of countably many dense open subsets of $Y$, such that $\kappa(X_y,D|_{X_y})\geq0$ for the fibre $X_y$ of $f$ over every closed point $y\in V$.

Let $X_\eta$ be the generic fibre of $f$, and assume that $h^0\big(X_\eta,\OO_{X_\eta}(mD)\big)=0$ for all positive integers $m$. By \cite[Theorem III.12.8]{Har77}, for each positive integer $m$ there exists a non-empty open subset $U_m\subseteq Y$ such that $h^0\big(X_y,\OO_{X_y}(mD)\big)=0$ for every point $y\in U_m$. But then for all $m$ and all $y\in V\cap \bigcap\limits_{m\in\N_{>0}}U_m $ we have $h^0\big(X_y,\OO_{X_y}(mD)\big)=0$, a contradiction. 

Therefore, there exists an effective divisor $G$ on $X_\eta$ such that $D|_{X_\eta}\sim_\Q G$, and (i) follows from \cite[Lemma 3.2.1]{BCHM}.

We now show (ii). By (i) there exists an effective 
$ \Q $-divisor $E$ on $X$ such that $D\sim_{\Q,Y}E$. Then $E|_F\sim_\Q 0$ for a very general fibre $F$ of $f$, hence $E|_F=0$ since $F$ is projective. Therefore, $E$ cannot be dominant over $Y$, and we set $U:=Y\setminus f(E)$.
\end{proof}

\subsection{G-pairs}

For the definitions and basic results on the singularities of pairs and the Minimal Model Program (MMP) we refer to \cite{KM98}. We explain below in more detail generalised pairs, abbreviated as g-pairs; for futher information we refer to \cite[Section 4]{BZ16}, and in particular to \cite[\S 2.1 and \S 3.1]{HanLi} for properties of dlt g-pairs. 

In this paper we concentrate on \emph{NQC g-pairs}, as they behave better in proofs; this was first realised in \cite{BZ16,HanLi}. The acronym NQC stands for \emph{nef $\Q$-Cartier combinations}. More precisely, an $\R$-divisor $D$ on a variety $X$ over $Z$ is an \emph{NQC divisor} if it is a non-negative linear combination of $\Q$-Cartier divisors on $X$ which are nef over $ Z $.

\begin{dfn}
	A \emph{generalised pair} or \emph{g-pair} $(X/Z,\Delta+M)$ consists of a normal variety $ X $ equipped with  projective morphisms 
		$$  X' \overset{f}{\longrightarrow} X \longrightarrow Z , $$ 
		where $ f $ is birational and $ X' $ is normal, $\Delta$ is an effective $ \R $-divisor on $X$, and $M'$ is an $\R$-Cartier divisor on $X'$ which is nef over $Z $ such that $ f_* M' = M $ and $ K_X + \Delta + M $ is $ \R $-Cartier. 
				
		Moreover, if $M'$ is an NQC divisor on $ X' $, then the g-pair $(X/Z,\Delta+M)$ is an \emph{NQC g-pair}.
\end{dfn}

For simplicity, we denote such a g-pair only by 
$ (X/Z,\Delta+M) $, but we implicitly remember the whole g-pair structure. Additionally, the definition is flexible with respect to $X'$ and $M'$: if $ g \colon Y \to X' $ is a projective birational morphism from a normal variety $ Y $, then we may replace $X'$ with $Y$ and $M'$ with $ g^*M'$. Hence we may always assume that $f \colon X'\to X$ in the above definition is a sufficiently high birational model of $ X $. 

A g-pair $ (X/Z,\Delta+M) $ is \emph{effective over $Z$} if there exists an effective divisor $G$ on $X$ such that $ K_X + \Delta + M \equiv_Z G$; it is \emph{pseudoeffective over $Z$} if the divisor $ K_X + \Delta + M $ is pseudoeffective over $Z $, that is, it is pseudoeffective on a very general fibre of the morphism $X\to Z$; it is \emph{log smooth} if $ X $ is smooth, with data $ X \overset{\operatorname{id}_X}{\longrightarrow} X \to Z $ and $ M' = M $, and if $ \Delta + M $ has SNC support.
 
\begin{dfn}
	Let $ (X,\Delta+M) $ be a g-pair which comes with data $ X' \overset{f}{\to} X \to Z $ and $ M' $. We can then write 
	$$ K_{X'} + \Delta' + M' \sim_\R f^* ( K_X + \Delta + M ) $$
	for some $ \R $-divisor $ \Delta' $ on $ X' $. Let $ E $ be a divisorial valuation over $ X $ which is a prime divisor on $ X' $, whose centre on $X$ is denoted by $c_X(E)$. The \emph{discrepancy of $ E $} with respect to $ (X,\Delta+M) $ is $ a (E, X, \Delta+M) := {-} \mult_E \Delta' $.
	The g-pair $ (X,\Delta+M) $ is: 
	\begin{enumerate}[label=(\alph*)]
		\item \emph{klt} if $a (E, X, \Delta+M) > -1 $ for all divisorial valuations $E$ over $X$,
		\item \emph{log canonical} if $a (E, X, \Delta+M) \geq -1 $ for all divisorial valuations $E$ over $X$,
		\item \emph{dlt} if it is log canonical, if there exists an open subset $U\subseteq X$ such that the pair $(U,\Delta|_U)$ is log smooth, and if $a(E,X,\Delta+M) = {-}1$ for some divisorial valuation $E$ over $X$, then the set $c_X(E)\cap U$ is non-empty and it is a log canonical center of $(U,\Delta|_U)$.
	\end{enumerate}
\end{dfn}

Note that the notions of klt and log canonical singularities for g-pairs are straightforward generalisations of the corresponding notions for usual pairs, whereas those of dlt singularities for g-pairs are very subtle. We adopted the definition from \cite{HanLi}, which behaves well under restrictions to log canonical centres and under operations of an MMP; such operations are analogous to those in the standard setting, see \cite[Section 4]{BZ16} or \cite[\S3.1]{HanLi} for the details. 

By \cite[Remark 2.3 and Lemma 3.5]{HanLi} one can run an MMP with scaling of an ample divisor over $Z$ for any $ \Q $-factorial log canonical g-pair $(X/Z,\Delta+M)$ such that $ (X,0) $ is klt, and in particular for any $\Q$-factorial dlt g-pair $(X/Z,\Delta+M)$; this fact is often used in this paper. 

\begin{rem}\label{rem:q-cartier}
Given a g-pair $ (X,\Delta+M) $, the divisor $ K_X + \Delta $ is $ \R $-Cartier if and only if the divisor $ M $ is 
$ \R $-Cartier. In this case, if the g-pair $ (X,\Delta+M) $ is klt (respectively log canonical), then the pair $ (X,\Delta) $ is klt (respectively log canonical), see \cite[Remark 4.2.3]{BZ16}. 
\end{rem}

\begin{lem}\label{lem:dltblowup}
Let $(X,\Delta +M)$ be a log canonical g-pair with data $ X' \overset{f}{\to} X \to Z $ and $ M' $. Then, after possibly replacing $f$ with a higher model, there exist a $\Q$-factorial dlt g-pair $\big(\widehat X,\widehat \Delta + \widehat M\,\big)$ with data $ X' \overset{g}{\to} \widehat X \to Z $ and $ M' $, and a projective birational morphism $\pi\colon \widehat X\to X$ such that 
$$K_{\widehat X} + \widehat\Delta +\widehat M \sim_\R \pi^*(K_X + \Delta +M).$$
The g-pair $\big(\widehat X,\widehat \Delta + \widehat M\,\big)$ is a \emph{dlt blowup} of $(X,\Delta +M)$.
\end{lem}

\begin{proof}
In the context of usual pairs this is \cite[Corollary 1.36]{Kol13}; in the formulation above this is \cite[Proposition 3.9]{HanLi}.
\end{proof}

\subsection{Minimal models}
\label{sec:MM} 
We distinguish between two types of minimal mo\-dels: minimal models in the usual sense and minimal models in the sense of Birkar-Shokurov. The former (often called \emph{log terminal models} in the lite\-rature) are those for which the map to the model is a birational contraction; the latter (often called \emph{log minimal models} in the literature) allow the map to the model to possibly extract divisors. Minimal models in the sense of Birkar-Shokurov are convenient for inductive purposes; however, our goal is always the existence of minimal models in the usual sense. In this paper, the phrase \emph{minimal model} always means a minimal model in the usual sense. 

\begin{dfn} \label{dfn:minmod}
	Let $(X/Z,\Delta +M)$ be a log canonical g-pair with data $ X' \overset{f}{\to} X \to Z $ and $ M' $. A birational map $ \varphi \colon (X,\Delta+M) \dashrightarrow (Y,\Delta_Y+M_Y)$ over $Z$ to a $\Q$-factorial dlt g-pair $ (Y/Z,\Delta_Y+M_Y) $ is a \emph{minimal model in the sense of Birkar-Shokurov over $Z$} of the g-pair $(X,\Delta+M)$ if $ \Delta_Y =\varphi_*\Delta+E$, where $E$ is the sum of all prime divisors which are contracted by $\varphi^{-1}$, if $M_Y=\varphi_*M$, if the divisor $K_Y+\Delta_Y+M_Y$ is nef over $Z$ and if
	$$a(F,X,\Delta+M) < a(F,Y,\Delta_Y+M_Y)$$
	for any prime divisor $ F $ on $ X $ which is contracted by $\varphi $.
		
	If, moreover, the map $\varphi$ is a birational contraction, $(Y,\Delta_Y+M_Y)$ is log canonical, but $Y$ is not necessarily $\Q$-factorial if $X$ is not $\Q$-factorial (and $Y$ is $\Q$-factorial if $X$ is $\Q$-factorial), then $\varphi$ is a \emph{minimal model} of $(X/Z,\Delta+M)$.
\end{dfn}

Note that, as for the usual pairs, the two notions of minimal models coincide if the g-pair $(X/Z,\Delta+M)$ is klt, see \cite[Remark 2.4]{Bir12b}. It has only recently been shown in \cite{HH19,HanLi} that the two notions coincide in much more general contexts. More precisely:

\begin{lem} \label{lem:MMcomparison}
	The following statements hold.
	\begin{enumerate}[label=(\roman*)]
		\item If $ (X/Z,\Delta) $ is a log canonical pair, then it has a minimal model over $Z$ if and only if it has a minimal model in the sense of Birkar-Shokurov over $Z$. 
		\item If $ (X/Z, \Delta+M) $ is a $\Q$-factorial NQC log canonical pair such that $ (X,0) $ is klt, then it has a minimal model over $Z$ if and only if it has a minimal model in the sense of Birkar-Shokurov over $Z$. 
	\end{enumerate}
\end{lem}

\begin{proof}
	We first show (i). If $ (Y/Z, \Delta_Y) $ is a minimal model of $ (X/Z, \Delta) $, then a dlt blow-up of $ (Y/Z, \Delta_Y) $ is a minimal model in the sense of Birkar-Shokurov of $ (X/Z, \Delta) $, see also \cite[Remark 2.4]{Hash18a}. The converse follows immediately by \cite[Theorem 1.7]{HH19}.
	
	We now show (ii). If $ (Y/Z, \Delta_Y + M_Y) $ is a minimal model of $ (X/Z, \Delta + M) $, then a dlt blow-up of $ (Y/Z, \Delta_Y + M_Y) $ is a minimal model in the sense of Birkar-Shokurov of $ (X/Z, \Delta + M) $ as above. The converse follows by repeating verbatim the proof of \cite[Theorem 1.7]{HanLi}, except that our assumption replaces \cite[Theorem 5.4]{HanLi} in that proof.
\end{proof}

\subsection{Weak Zariski decompositions}

The notion of a weak Zariski decomposition is crucial in this paper. We concentrate only on its NQC version, as it behaves better in proofs.

\begin{dfn}\label{dfn:NQCWZD}
	Let $X$ be a normal quasi-projective variety over $Z$ and let $ D $ be an $ \R $-Cartier divisor on $ X $. An \emph{NQC weak Zariski decomposition} of $ D $ over $ Z $ consists of a projective birational morphism $ f \colon W \to X $ from a normal variety $ W $ and a numerical equivalence $ f^* D \equiv_Z P + N $ such that $ P $ is an NQC divisor and $ N $ is an effective $\R$-Cartier divisor on $W$.
 
	We say that a g-pair $ (X/Z,\Delta+M) $ admits an NQC weak Zariski decomposition if the divisor $ K_X + \Delta + M $ has an NQC weak Zariski decomposition over $ Z $. 
\end{dfn}

\begin{rem}\label{rem:biratWZD}
With the same notation as in Definition \ref{dfn:NQCWZD}, assume that $D$ has an NQC weak Zariski decomposition over $ Z $ and consider a projective surjective morphism $g\colon Y\to X$ from a normal variety $ Y $. Then $g^* D$ has an NQC weak Zariski decomposition over $ Z $: indeed, by considering a resolution of indeterminacies $(p,q)\colon T \to Y \times W $ of $f^{-1} \circ g \colon Y \dashrightarrow W$ such that $T$ is normal, we have
$$	p^* g^* D = q^* f^* D \equiv_Z q^*P + q^*N, $$
which proves the assertion. 

If $g$ is additionally birational, then the converse is also clear, namely if $g^*D$ has an NQC weak Zariski decomposition over $Z$, then $D$ has an NQC weak Zariski decomposition over $Z$.

Consequently, a log canonical g-pair $ (X/Z,\Delta+M) $ admits an NQC weak Zariski decomposition if and only if any dlt blowup of $ (X/Z,\Delta+M) $ admits an NQC weak Zariski decomposition. 
\end{rem}

\begin{rem} \label{rem:MMimplWZD}
	Let $ (X/Z,\Delta+M) $ be an NQC log canonical g-pair. If $ (X,\Delta+M) $ has a minimal model in the sense of Birkar-Shokurov over $Z$, then it admits an NQC weak Zariski decomposition over $Z$ by \cite[Proposition 5.1]{HanLi}. 
\end{rem}

\begin{rem}\label{rem:addM}
Let $D$ be an $\R$-Cartier divisor on a variety $X/Z$ which has an NQC weak Zariski decomposition over $Z$, and let $M$ be an $\R$-Cartier divisor on $X$ which is the pushforward of an NQC divisor on some birational model of $X$. Then $ D+M $ has an NQC weak Zariski decomposition over $Z$. Indeed, let $f\colon Y\to X$ be a projective birational morphism from a normal variety $Y$ such that $ f^* D \equiv_Z P + N $, where $ P $ is an NQC divisor and $ N \geq 0 $. We may assume that $M=f_*M'$, where $M'$ is an NQC divisor on $Y$. By the Negativity lemma \cite[Lemma 3.39]{KM98} we have $f^*M=M'+E$, where $E$ is an effective $f$-exceptional $ \R $-Cartier $\R$-divisor on $Y$. Hence 
$$ f^* (D+M) \equiv_Z (P+M') + (N+E) , $$ which proves the claim. 
\end{rem}

We need the following basic result which illustrates how weak Zariski decompositions behave under the operations of an MMP.

\begin{lem}\label{lem:updownWZD}
	Let $f\colon X\dashrightarrow Y$ be a birational contraction between $\Q$-factorial varieties which are projective over $Z$. Let $D$ be an $\R$-divisor on $X$ such that the map $f$ is $D$-non-positive. Then $D$ has an NQC weak Zariski decomposition over $Z$ if and only if $f_*D$ has an NQC weak Zariski decomposition over $ Z $.
\end{lem}

\begin{proof}
	Set $G=f_*D$. If $(p,q)\colon W\to X\times Y$ is a resolution of indeterminacies of $f$, then there exists an effective $q$-exceptional $\R$-Cartier $\R$-divisor $E$ on $W$ such that $p^*D\equiv_Z q^*G+E$. By Remark \ref{rem:biratWZD} we may replace $X$ by $W$, and hence we may assume that $f$ is a morphism and
	\begin{equation}\label{eq:23-1}
		D\equiv_Z f^*G+E.
	\end{equation}
	
	Assume that $D$ has an NQC weak Zariski decomposition. Then there exists a projective birational morphism $\pi\colon \widetilde{X} \to X$ from a normal variety $ \widetilde{X} $ such that $\pi^*D\equiv_Z P+N$, where $P$ is an NQC divisor and $N\geq0$. If we set $g:=f\circ\pi$, then $G\equiv_Z g_*P+g_*N$ by \eqref{eq:23-1}. By the Negativity lemma \cite[Lemma 3.39]{KM98} we have $g^*g_*P=P+F$ for some effective $g$-exceptional $ \R $-Cartier $\R$-divisor $F$ on $\widetilde{X}$, hence 
	$$g^*G \equiv_Z g^*g_*P+g^*g_*N=P+(F+g^*g_*N),$$
	which shows that $G$ has an NQC weak Zariski decomposition over $ Z $.
	
	Conversely, if $G$ has an NQC weak Zariski decomposition over $ Z $, then we conclude similarly as in Remark \ref{rem:biratWZD}. 
\end{proof}

We will need the following lemma in the proofs of various results announced in the introduction.

\begin{lem}\label{lem:downthedimension}
	Let $n\geq k$ be positive integers. Assume that every pseudoeffective smooth pair $(X,0)$ over $Z$ of dimension $n$ admits an NQC weak Zariski decomposition over $Z$. Then every pseudoeffective smooth pair $(Y,0)$ over $Z$ of dimension $k$ admits an NQC weak Zariski decomposition over $Z$.
\end{lem}

\begin{proof}
	We adopt the idea from \cite[Lemma 3.2]{Hash18b}. Let $(Y,0)$ be a $k$-dimensional pair as in the statement of the lemma, and set $X:=Y\times A$, where $A$ is any $(n-k)$-dimensional abelian variety. Let $p\colon X\to Y$ and $q\colon X\to A$ be the projection maps. Then $K_X\sim p^* K_Y$, and in particular, $K_X$ is pseudoeffective. By assumption there exists a projective birational morphism $f \colon W \to X$ from a normal variety $W$ such that
	\begin{equation}\label{eq:23-2}
		f^* K_X \equiv_Z P+N,
	\end{equation}
	where $P$ is an NQC divisor and $N\geq0$. Now, let $F$ be a very general fibre of $q$ and set $ F_W:=f^{-1}(F) $; note that $F\simeq Y$. By restricting \eqref{eq:23-2} to $ F_W $ we obtain $ \big( f |_{F_W} \big)^* K_F \equiv_Z P |_{F_W} + N|_{F_W} $, which gives an NQC weak Zariski decomposition of $ K_Y $ over $ Z $. 
\end{proof}

The next result is an immediate corollary of Lemma \ref{lem:downthedimension} and of Theorems \ref{thm:HH_gMM} and \ref{thm:HMref}; it will not be used in the rest of the paper.

\begin{lem}
	Let $n\geq k$ be positive integers. Assume that every pseudoeffective smooth pair $(X,0)$ over $Z$ of dimension $n$ has a minimal model over $Z$. Then every pseudoeffective smooth pair $(Y,0)$ over $Z$ of dimension $k$ has a minimal model over $Z$.
\end{lem}

\begin{proof}
	Let $ (Y,0) $ be a pseudoeffective smooth pair over $Z$ of dimension $k$. By assumption and Remark \ref{rem:MMimplWZD}, every pseudoeffective smooth pair $(X,0)$ over $Z$ of dimension $n$ admits an NQC weak Zariski decomposition over $Z$. By Lemma \ref{lem:downthedimension} and Theorem \ref{thm:HMref}, every NQC log canonical g-pair over $ Z $ of dimension at most $ n $ admits an NQC weak Zariski decomposition over $ Z $. In particular, $ (Y,0) $ admits an NQC weak Zariski decomposition over $Z$ and hence has a minimal model over $Z$ by Theorem \ref{thm:HH_gMM}.
\end{proof}

\subsection{On effectivity}

We prove a criterion for a g-pair to be relatively effective in Corollary \ref{cor:relativeMMPcontracting} below. We use the definition of the \emph{relatively movable cone} $\overline{\Mov}(X/Y)$ of a projective morphism $f\colon X\to Y$, see \cite[Definition 2.1]{Fuj11b}; we then say that an $\R$-Cartier $\R$-divisor $D$ on $X$ is \emph{movable over $Y$} if the class of $D$ belongs to $\overline{\Mov}(X/Y)$.

We first need the following result, cf.\ \cite[Theorem 2.3]{Fuj11b}.

\begin{lem}\label{lem:easytermination}
	Let $(X/Z,\Delta+M)$ be a $\Q$-factorial log canonical g-pair such that $ (X,0) $ is klt and $ K_X + \Delta + M $ is pseudoeffective over $Z$. Assume that we have a $(K_X+\Delta+M)$-MMP with scaling of an ample divisor $A$ over $Z$. Then on some variety $X_i$ in this MMP, the strict transform of $K_X+\Delta+M$ becomes movable over $Z$. In particular, the restriction of this strict transform to a very general fibre of the induced morphism $X_i\to Z$ is movable.
\end{lem}

\begin{proof}
	The proof is analogous to the proof of \cite[Theorem 2.3]{Fuj11b} and has already appeared implicitly in the proof of \cite[Proposition 3.8]{HanLi}. Nevertheless, we include the details for the benefit of the reader.
	
	First of all, we may run such an MMP by \cite[Lemma 3.5]{HanLi}. Let $(X_i/Z,\Delta_i+M_i)$ be the g-pairs in the MMP, let $A_i$ be the strict transform of $A$ on $X_i$, and set $\nu_i := \inf\{t\in\R_{\geq 0} \mid K_{X_i} +\Delta_i+M_i+tA_i\text{ is nef over }Z\}$ and $\nu:=\lim \limits_{i \to \infty} \nu_i$.
	
	If $\nu> 0$, then the given MMP is also an MMP on $K_X +\Delta + M + \frac{\nu}{2} A$. Thus, by \cite[Lemma 3.5]{HanLi} there exists a boundary divisor $B$ such that 
	$K_X+B \sim_{\R,Z} K_X + \Delta + M + \frac{\nu}{2} A  $, 
	the pair $(X,B)$ is klt and $ B $ is big over $ Z $. By \cite[Corollary 1.4.2]{BCHM}, the $(K_X+B)$-MMP with scaling of $ A $ over $ Z $ terminates, and therefore the original MMP terminates. 
	
	If $\nu= 0$, then we assume that the MMP does not terminate (otherwise the assertion is clear), and that the strict transform of $K_X+\Delta+M$ never becomes movable over $Z$. We may also assume that the MMP consists only of flips. For each $i$, let $H_i$ be a divisor on $X_i$ which is ample over $Z$ and such that, if $H_{i,X}$ is the strict transform of $H_i$ on $X$, then $\lim\limits_{i\to\infty}[H_{i,X}]=0$ in $N^1(X/Z)$. Since $K_{X_i}+\Delta_i+M_i+\nu_i A_i+H_i$ is ample over $Z$ for every $i$, the strict transform $K_X+\Delta+M+\nu_i A+H_{i,X}$ is movable over $Z$ for every $i$, and therefore $K_X+\Delta+M$ is movable over $Z$, a contradiction. 
	
	The last claim in the lemma follows easily from the first one.
\end{proof}

\begin{cor}\label{cor:relativeMMPcontracting}
	Let $(X,\Delta+M)$ be a quasi-projective $\Q$-factorial dlt g-pair and let $f\colon X\to Y$ be a projective surjective morphism to a normal quasi-projective variety $Y$. Assume that $\nu\big(F,(K_X+\Delta+M)|_F\big)=0$ and $h^1(F,\OO_F)=0$ for a very general fibre $F$ of $f$. Then $K_X+\Delta+M$ is effective over $Y$.
\end{cor}

\begin{proof}
	By \cite[Lemma 3.5]{HanLi} we may run a $(K_X+\Delta+M)$-MMP with scaling of an ample divisor over $Y$. By Lemma \ref{lem:easytermination}, after finitely many steps of this MMP we obtain a $(K_X+\Delta+M)$-negative map $\theta\colon X\dashrightarrow X'$ such that $(K_{X'}+\Delta'+M')|_{F'}$ is movable, where $F'$ is a very general fibre of the induced morphism $f'\colon X'\to Y$, and $\Delta':=\theta_*\Delta$ and $M':=\theta_*M$. Additionally, since $\nu\big(F,(K_X+\Delta+M)|_F\big)=0$ by assumption, it follows by \cite[\S2.2]{LP20a} that $\nu\big(F',(K_{X'}+\Delta'+M')|_{F'}\big)=0$. Hence $(K_{X'}+\Delta'+M')|_{F'}\equiv0$ by \cite[Propositions III.1.14(1) and V.2.7(8)]{Nak04}.
	
	By Lemma \ref{lem:numtriv} there exists an open subset $U\subseteq Y$ such that 
	$$(K_{X'}+\Delta'+M')|_{U'}\equiv_U0,$$
	where $U'=f'^{-1}(U)$. The natural projection $\Div_\R(U')\to N^1(U'/U)_\R$ is defined over $\Q$, so there exist $\Q$-divisors $D_1,\dots,D_m$ on $U'$ such that $D_i\equiv_U 0$ for each $ i \in \{1, \ldots, m\} $ and real numbers $r_1,\dots,r_m$ such that 
	\begin{equation}\label{eq:36}
	(K_{X'}+\Delta'+M')|_{U'}=\sum\limits_{i=1}^m r_iD_i.
	\end{equation}
	Now, since $ X $ is a klt variety by \cite[Remark 2.3]{HanLi}, so is $ F $, and thus $F$ has rational singularities. The same also holds for $F'$, because $X'$ is a klt variety by \cite[Lemma 3.7 and Remark 2.3]{HanLi}. As $h^1(F,\OO_{F})=0$ by assumption, we therefore obtain $h^1(F',\OO_{F'})=0$. We deduce that $D_i|_{F'}\sim_\Q0$ for every $i \in \{1, \ldots, m\}$, and, after possibly shrinking $U$, by Lemma \ref{lem:Grauert} we infer that $D_i\sim_{\Q,U} 0$. But then \eqref{eq:36} yields
	$$(K_{X'}+\Delta'+M')|_{U'}\sim_{\R,U}0.$$
	It follows now by \cite[Lemma 3.2.1]{BCHM} that $K_{X'}+\Delta'+M'$ is effective over $Y$, and therefore so is $K_X+\Delta+M$.
\end{proof}

\subsection{MMPs as sequences of flops}

We need the following two results that exploit the boundedness of extremal rays. 

\begin{lem} \label{lem:trivMMP}
	Let $ (X/Z,\Delta) $ be a log canonical pair such that $ K_X + \Delta $ is nef over $Z$. Then there exists $ \varepsilon_0 >0 $ such that for every $ 0<\varepsilon <\varepsilon_0$, any $ \big(K_X + (1-\varepsilon) \Delta\big) $-MMP over $Z$ is $ (K_X + \Delta) $-trivial. 
\end{lem}

\begin{proof}
This is a special case of \cite[Proposition 3.2(5)]{Bir11}; note that the proof of that result works also for log canonical pairs, see \cite[Remark 18.8]{Fuj11a}.
\end{proof}

The following is the analogue of the previous result in the context of g-pairs; similar statements were also observed in \cite{HM18} and \cite[\S 3.3]{HanLi}.

\begin{lem} \label{lem:trivMMPg}
	Let $ (X/Z,\Delta+M) $ be a $ \Q $-factorial NQC log canonical g-pair such that $ (X,0) $ is klt. Assume that $ K_X + \Delta +M$ is nef over $Z$. Then there exists $ \varepsilon_0 >0 $ such that for every $ 0<\varepsilon <\varepsilon_0$, any $ \big(K_X + \Delta+(1-\varepsilon) M\big) $-MMP with scaling of an ample divisor over $Z$ is $ (K_X + \Delta+M) $-trivial. 
\end{lem}

\begin{proof}
	By \cite[Proposition 3.16]{HanLi} there exist $\Q$-divisors $\Delta_1,\dots,\Delta_m$ and $M_1,\dots,M_m$ such that each g-pair $(X/Z,\Delta_i+M_i)$ is log canonical with
	$ K_X + \Delta_i + M_i $ nef over $Z$, and there exist positive real numbers $\alpha_1,\dots,\alpha_m$ such that $\sum \alpha_i=1$ and
	\begin{equation}\label{eq:35}
	K_X+\Delta+M=\sum\limits_{i=1}^m \alpha_i(K_X+\Delta_i+M_i).
	\end{equation}
	Fix a positive integer $ r $ such that $ r(K_X+\Delta_i+M_i) $ is Cartier for each $i \in \{1,\dots,m\}$. Consider the set 
	$$\mathcal S=\big\{\textstyle\sum\alpha_i n_i>0\mid n_i\in\Z\text{ and } n_i\geq0\big\}.$$ 
	Then clearly there exists $\beta>0$ such that $s>\beta$ for all $s\in\mathcal S$. Set 
	$$ \varepsilon_0 := \frac{\beta}{\beta+2r\dim X} $$
	and fix any $ 0<\varepsilon <\varepsilon_0 $. 
	
	By \cite[Lemma 3.5]{HanLi} we may run a 
	$ \big(K_X + \Delta+(1-\varepsilon) M\big) $-MMP with scaling of an ample divisor over $Z$. It suffices to show that this MMP is $ (K_X + \Delta_i+M_i) $-trivial at the first step for each $i \in \{1,\dots,m\}$, since then the strict transform of $K_X+\Delta_i+M_i$ stays nef and $ r(K_X+\Delta_i+M_i) $ stays Cartier along the MMP due to \cite[Theorem 3.25(4)]{KM98}. 
	
	Let $ R $ be a $ \big(K_X +\Delta+ (1-\varepsilon) M\big) $-negative extremal ray over $Z$. Since $(K_X+\Delta+M)\cdot R\geq0$ by assumption, we infer 
	$$(K_X+\Delta)\cdot R<0\quad\text{and}\quad M\cdot R>0.$$
	By the boundedness of extremal rays \cite[Proposition 3.13]{HanLi} we may find a curve $ C $ on $ X $ whose class is in $ R $ such that
	\begin{equation}\label{eq:45}
	-2 \dim X\leq (K_X+\Delta)\cdot C<0
	\end{equation}
	and
	\begin{equation}\label{eq:46}
	-2 \dim X \leq \big( K_X + \Delta+(1-\varepsilon) M\big) \cdot C < 0 .	
	\end{equation}
	From \eqref{eq:45} and \eqref{eq:46} we obtain $(1-\varepsilon)M\cdot C\leq 2\dim X$. Therefore
	$$0 < M \cdot C \leq \frac{2 \dim X}{1-\varepsilon} .$$
	This implies
	$$(K_X + \Delta +M )\cdot C = \big( K_X + \Delta+(1-\varepsilon) M\big) \cdot C + \varepsilon M \cdot C < \frac{2\varepsilon \dim X}{1-\varepsilon}< \frac{\beta}{r}. $$
	If $(K_X + \Delta +M )\cdot C>0$, then $r(K_X + \Delta +M )\cdot C \in\mathcal S$ by \eqref{eq:35}, a contradiction to the choice of $\beta$. Hence 
	$( K_X + \Delta+M ) \cdot C=0$, and by \eqref{eq:35} we obtain $ (K_X+\Delta_i+M_i) \cdot C = 0 $, which finishes the proof.
\end{proof}

\subsection{On termination with scaling}
\label{sec:termination}

We need a couple of results on the termination of some special MMPs. 

In \cite{HanLi}, Han and Li show that for $\Q$-factorial NQC dlt g-pairs, the
existence of NQC weak Zariski decompositions is equivalent to the existence
of minimal models. We need a slightly reformulated version of their main
result: the difference is that we assume only that \emph{the given pair} has an NQC weak Zariski decomposition, and not every such pair. Up to a reorganisation of the proof, the following is \cite[Theorem 1.7]{HanLi}.

\begin{thm}\label{thm:HanLi}
	Assume the existence of NQC weak Zariski decompositions for NQC $\Q$-factorial dlt g-pairs of dimensions at most $n-1$. 
	
	Let $ (X/Z,\Delta+M) $ be an NQC log canonical g-pair of dimension $ n $ which has an NQC weak Zariski decomposition over $Z$. Then the following statements hold.
\begin{enumerate}
\item[(i)] The g-pair $ (X,\Delta+M) $ has a minimal model in the sense of Birkar-Shokurov over $Z$.
\item[(ii)] If additionally $X$ is $ \Q $-factorial and $(X,0)$ is klt, then any $(K_X + \Delta +M)$-MMP with scaling of an ample divisor over $Z$ terminates.
\end{enumerate}	
\end{thm}

\begin{proof}
	In this sketch of the proof, we follow closely the proof of \cite[Theorem 1.7]{HanLi}, and only indicate how that proof can be reorganised in order to obtain the statement above.
	
	We first show (ii). By the proof of \cite[Theorem 1.7]{HanLi}, to prove (ii) it suffices to show that $(X,\Delta+M)$ has a minimal model in the sense of Birkar-Shokurov over $ Z $. 
	
	Let $\mathcal S$ be the set of all NQC log canonical g-pairs of dimension $n$ whose underlying variety is $ \Q $-factorial klt and which have an NQC weak Zariski decomposition but do not have a minimal model in the sense of Birkar-Shokurov. If $\mathcal{S} = \emptyset$, then we are done.
	
	Otherwise, to each element $(Y/Z,\Delta_Y+M_Y)\in\mathcal S$ with an NQC weak Zariski decomposition $ g_Y^* (K_Y+\Delta_Y+M_Y) \equiv_Z P_Y+N_Y$, where $g_Y\colon Y'\to Y$ is a birational model, we can associate a quantity $\theta(Y/Z,\Delta_Y+M_Y,N_Y)$ as in \cite[Definition 5.2]{HanLi}; for convenience, we recall that this is the number of components of $ (g_Y)_* N_Y $ which are not components of $ \lfloor \Delta_Y \rfloor $. Next, as in Step 1 of the proof of \cite[Theorem 5.4]{HanLi}, we may assume that we have chosen an element $(Y/Z,\Delta_Y+M_Y)\in\mathcal S$ which is an NQC log smooth g-pair with an NQC weak Zariski decomposition $K_Y+\Delta_Y+M_Y\equiv_Z P_Y+N_Y$ that minimises this non-negative integer $ \theta $ in the set $\mathcal S$. Now, we distinguish two cases: if $\theta(Y/Z,\Delta_Y+M_Y,N_Y)=0$, then a contradiction follows by Step 2 of the proof of \cite[Theorem 5.4]{HanLi}, while if $\theta(Y/Z,\Delta_Y+M_Y,N_Y)>0$, then a contradiction follows by Step 3 of the proof of \cite[Theorem 5.4]{HanLi}. Hence $\mathcal{S} = \emptyset$, as desired, and this concludes the proof of (ii).
	
	For (i), as in Step 1 of the proof of \cite[Theorem 5.4]{HanLi} we may assume that $(X,\Delta+M)$ is log smooth, and then the result follows from (ii).
\end{proof}

Combining the above with one of the main results of \cite{HH19}, we obtain:

\begin{thm}\label{thm:HH_gMM}
	Assume the existence of NQC weak Zariski decompositions for $ \Q $-factorial NQC dlt g-pairs of dimensions at most $n-1$. 

Let $ (X/Z,\Delta) $ be a log canonical pair of dimension $ n $. If $ (X,\Delta) $ has an NQC weak Zariski decomposition over $Z$, then there exists a $(K_X + \Delta)$-MMP with scaling of an ample divisor over $Z$ which terminates. 
\end{thm}

\begin{proof}
By \cite[Theorem 1.7]{HH19} it suffices to show that the pair $(X,\Delta)$ has a minimal model in the sense of Birkar-Shokurov over $Z$. We conclude by Theorem \ref{thm:HanLi}(i).
\end{proof}

The next result is a special case of \cite[Lemma 4.3]{HanLiu}, and builds on \cite[Proposition 8.7]{DHP13}, \cite[Lemma 3.1]{Gon15} and \cite[Theorem 3.3]{DL15}. We include a detailed proof for the benefit of the reader and we follow the presentation in \cite{DL15}.

\begin{lem}\label{lem:Gongyo}
	Let $(X/Z,\Delta+M)$ be a pseudoeffective NQC $\Q$-factorial dlt g-pair of dimension $n$ such that for each $\varepsilon>0$ the divisor $K_X+\Delta+(1-\varepsilon)M$ is not pseudoeffective over $Z$. Then there exists a birational contraction $\varphi\colon X\dashrightarrow X'$ over $ Z $ and a fibration $f\colon X'\to Y$ over $ Z $ such that, if we set $\Delta':=\varphi_*\Delta$ and $M':=\varphi_*M$, then:
	\begin{enumerate}
		\item[(i)] $(X',\Delta'+M')$ is a $\Q$-factorial log canonical g-pair, 
		\item[(ii)] $K_{X'}+\Delta'+M'\sim_{\R,Y}0$, 
		\item[(iii)] $\varphi$ is a $\big(K_X+\Delta+(1-\varepsilon)M\big)$-MMP for some $0<\varepsilon\ll1$ and $f$ is the corresponding Mori fibre space.
	\end{enumerate}
\end{lem}

\begin{proof}
	Fix an effective divisor $A$ on $X$ which is ample over $Z$. Let $\{ \varepsilon_i \} $ be a decreasing sequence of positive real numbers such that $\lim\limits_{i\to\infty}\varepsilon_i=0$. Set 
	$$y_i=\inf\big\{t\in\R_{\geq0}\mid K_X+\Delta+(1-\varepsilon_i)M+tA\text{ is pseudoeffective over }Z\big\}$$
	and note that $y_i>0$ for all $i$.
		
	Fix $i$. By \cite[Lemma 3.5]{HanLi} we may run a $\big(K_X+\Delta+(1-\varepsilon_i)M\big)$-MMP with scaling of $A$ over $Z$. Note that this MMP is also an MMP on $K_X +\Delta + (1-\varepsilon_i)M + \nu A$ for some $0<\nu< y_i$. By \cite[Lemma 3.5]{HanLi} there exists a boundary divisor $B$ such that $K_X + \Delta + (1-\varepsilon_i)M + \nu A \sim_{\R,Z} K_X+B$ and the pair $(X,B)$ is klt. Therefore this MMP is clearly also a $(K_X+B)$-MMP with scaling of $ A $ over $ Z $, which terminates with a Mori fibre space $g_i\colon X_i\to Y_i$ over $ Z $ by \cite[Corollary 1.3.3]{BCHM}. Let $f_i\colon X\dashrightarrow X_i$ be the resulting birational contraction over $ Z $ and denote by $\Delta_i$, $M_i$ and $A_i$ the strict transforms of $\Delta$, $M$ and $A$ on $X_i$. 
%
%
	Since $ K_{X_i} + \Delta_i + M_i $ is pseudoeffective over $ Z $, there exist effective $\R$-divisors $E_j$ on $X_i$ such that $\lim\limits_{j\to\infty}[E_j]=[K_{X_i} + \Delta_i + M_i]$ in $N^1(X_i/Z)_\R$. Let $C$ be a curve on $X_i$ which does not belong to $ \bigcup\Supp E_j $ and is contracted by $g_i$. Then
	$$\big(K_{X_i} + \Delta_i + (1-\varepsilon_i)M_i\big) \cdot C < 0$$
	by the definition of the MMP, and 
	$$(K_{X_i} + \Delta_i + M_i) \cdot C \geq 0 $$
	by the choice of $C$. Thus, there exists $\eta_i\in(1-\varepsilon_i,1]$ such that $(K_{X_i}+\Delta_i+\eta_iM_i)\cdot C=0$, and hence 
	$$K_{X_i}+\Delta_i+\eta_iM_i\equiv_{Y_i}0 ,$$
	since all contracted curves are numerically proportional. In particular, if $F_i$ is a very general fibre of $g_i$ and if we set $\Delta_{F_i}=\Delta_i|_{F_i}$ and $M_{F_i}=M_i|_{F_i}$, then 
	$$K_{F_i}+\Delta_{F_i}+\eta_iM_{F_i}\equiv 0.$$
	
	For every $ i $, set 
	$$\tau_i=\max\big\{t\in\R_{\geq 0}\mid (F_i,\Delta_{F_i}+tM_{F_i})\text{ is log canonical}\big\}$$
	and note that $1-\varepsilon_i \leq \tau_i$ since each g-pair  $(F_i,\Delta_{F_i}+(1-\varepsilon_i)M_{F_i})$ is log canonical. If $(F_i,\Delta_{F_i}+M_{F_i})$ is not log canonical for infinitely many $i$, then after passing to a subsequence we can assume that $\tau_i < 1$ for all $i$, and since $1-\varepsilon_i \leq \tau_i$ and $\lim\limits_{i\to\infty} (1-\varepsilon_i) = 1$, we can assume that the sequence $\{\tau_i\}$ is strictly increasing, which contradicts \cite[Theorem 1.5]{BZ16}. Therefore $(F_i,\Delta_{F_i}+M_{F_i})$ is log canonical for $i\gg0$. It follows by \cite[Theorem 1.6]{BZ16} that the sequence $\{\eta_i\}$ is eventually constant, and thus $\eta_i=1$ for $i\gg0$. Hence, we take $\varphi$ to be any of the $f_i$ for $i\gg0$, and we set $X':=X_i$ and $Y:=Y_i$.
\end{proof}

\section{On weak Zariski decompositions}	\label{sec:main}

In this section we prove the technical core of the paper. The first two results show that, modulo a suitable assumption on g-pairs in lower dimensions, we may sometimes deduce the existence of weak Zariski decompositions. The proof of Theorem \ref{thm:HM} follows the same strategy as that of \cite[Theorem 2]{HM18}, with additional complications due to the fact that we are working with divisors with real coefficients. The proof of Theorem \ref{thm:maintechnical} follows similar ideas in the context of usual pairs. The main technical input for both proofs comes from \cite{DHP13,Gon15,DL15}. The next two results are important consequences of the aforementioned theorems and allow one to reduce the inductive assumption on g-pairs to a statement about usual pairs. 
At the end of this section we prove Theorem \ref{thm:HMref}.

\begin{thm}\label{thm:maintechnical}
	Assume the existence of NQC weak Zariski decompositions for NQC log canonical g-pairs of dimensions at most $n-1$. 
	
	Let $(X/Z,\Delta)$ be a pseudoeffective $\Q$-factorial dlt pair of dimension $n$ such that for each $\varepsilon>0$ the divisor $K_X+(1-\varepsilon)\Delta$ is not pseudoeffective over $Z$. Then $(X,\Delta)$ has an NQC weak Zariski decomposition over $ Z $. 
\end{thm}

\begin{proof}
	We proceed in four steps. 
	\medskip
	
	\emph{Step 1.}
	In this step we show that we may assume the following:
	
	\medskip
	
	\emph{Assumption 1.}
	There exists a fibration $\xi\colon X\to Y$ over $Z$ to a normal quasi-projective variety $Y$ such that $\dim Y<\dim X$ and such that:
	\begin{enumerate}
		\item[(a$_1$)] $\nu\big(F,(K_X+\Delta)|_F\big)=0$ and $h^1(F,\OO_F)=0$ for a very general fibre $F$ of $\xi$, 
		\item[(b$_1$)] $K_X+(1-\varepsilon)\Delta$ is not $\xi$-pseudoeffective for any $\varepsilon>0$. 
	\end{enumerate}
	
	\medskip
	
	To this end, pick a decreasing sequence $ \left\{ \varepsilon_i \right\} $ of positive numbers such that $\lim\limits_{i\to\infty}\varepsilon_i=0$. By \cite[Lemma 3.1]{Gon15}\footnote{This lemma works in the relative setting.} applied to the divisors $\Delta_i:=(1-\varepsilon_i)\Delta$, there exists a birational contraction $ \varphi\colon X\dashrightarrow S$ over $Z$ and a fibration $f\colon S\to Y$ over $Z$ such that, if we set $\Delta_S:=\varphi_*\Delta$, then:
	\begin{enumerate}
		\item[(a)] $ (S,\Delta_S) $ is a $\Q$-factorial log canonical pair, 
		\item[(b)] $ Y $ is a normal quasi-projective variety with $\dim Y<\dim X$,
		\item[(c)] $ K_S + \Delta_S \equiv_Y 0 $,
		\item[(d)] $\varphi$ is a $\big(K_X+(1-\varepsilon_i)\Delta\big)$-MMP for some $i\gg0$ and $f$ is the corresponding Mori fibre space.
	\end{enumerate}
	
	Let $(p,q)\colon W\to X\times S$ be a smooth resolution of indeterminacies of $\varphi$. We may write
	\begin{equation}\label{eq:onW}
	K_W+\Delta_W\sim_\R p^*(K_X+\Delta)+E,
	\end{equation}
	where the divisors $\Delta_W$ and $E$ are effective and have no common components. By passing to a higher model we may assume that the pair $(W,\Delta_W)$ is log smooth. 
	\begin{center}
		\begin{tikzcd}
			& W \arrow[dl, "p" swap] \arrow[dr, "q"] && \\
			X \arrow[rr, dashed, "\varphi" ] && S \arrow[r, "f"] & Y 
		\end{tikzcd}
	\end{center} 
	
	Let $F$ be a very general fibre of $f$ and set $F_W=q^{-1}(F)\subseteq W$, $\Delta_F:=\Delta_S|_F$ and $\Delta_{F_W}:=\Delta_W|_{F_W}$. Then the divisors $K_F+\Delta_F$ and $K_{F_W}+\Delta_{F_W}$ are pseudoeffective and we have $(q|_{F_W})_*(K_{F_W}+\Delta_{F_W})=K_F+\Delta_F$. Hence, by \cite[Lemma 3.1]{DL15} and by (c),
	$$\nu(F_W,K_{F_W}+\Delta_{F_W})\leq\nu(F,K_F+\Delta_F)=0,$$
	hence $\nu(F_W,K_{F_W}+\Delta_{F_W})=0$.
		
	Moreover, for every $\varepsilon>0$, the divisor $K_{F_W}+(1-\varepsilon)\Delta_{F_W}$ is not pseudoeffective: otherwise, the divisor $K_F+(1-\varepsilon)\Delta_F=(q_{F_W})_*\big(K_{F_W}+(1-\varepsilon)\Delta_{F_W}\big)$ would be pseudoeffective for some $\varepsilon>0$, a contradiction to (c) and (d).
		
	Furthermore, since $S$ is a klt variety by (d), so is $F$, and hence $F$ has rational singularities. In addition, $h^1(F,\OO_F)=0$ by (d) and by the Kodaira vanishing theorem. It follows that $h^1(F_W,\OO_{F_W})=0$.
	
	If $K_W+\Delta_W$ has an NQC weak Zariski decomposition over $ Z $, then $K_X+\Delta$ has an NQC weak Zariski decomposition  over $ Z $ by Lemma \ref{lem:updownWZD}. 
	
	Therefore, by replacing $(X,\Delta)$ by $(W,\Delta_W)$ and by setting $\xi:=f\circ q$, we achieve Assumption 1.
	
	\medskip
	
	\emph{Step 2.}
	If $\dim Y=0$ (and thus necessarily $\dim Z=0$), then $K_X+\Delta\equiv N_\sigma(K_X+\Delta)$ by \cite[Proposition V.2.7(8)]{Nak04}. Hence $K_X+\Delta$ has an NQC weak Zariski decomposition, and we are done. 
	
	\medskip
	
	\emph{Step 3.}
	Assume from now on that $\dim Y>0$. In this step we show that we may assume the following:
	
	\medskip
	
	\emph{Assumption 2.}
	There exists a fibration $g\colon X\to T$ to a normal quasi-projective variety $T$ such that:
	\begin{enumerate}
	\item[(a$_2$)] $0<\dim T<\dim X$, 
	\item[(b$_2$)] $K_X+\Delta\equiv_T 0$,
	\item[(c$_2$)] the numerical equivalence over $T$ coincides with the $\R$-linear equi\-valence over $T$.
	\end{enumerate} 
\emph{However, instead of the pair $(X,\Delta)$ being $\Q$-factorial dlt, we may only assume that it is a $\Q$-factorial log canonical pair and $(X,0)$ is klt.}
	
	\medskip
	
	To this end, by Assumption 1 and by Corollary \ref{cor:relativeMMPcontracting} the divisor $K_X+\Delta$ is effective over $Y$. By assumptions of the theorem and by Theorem \ref{thm:HanLi} we may run a $(K_X+\Delta)$-MMP with scaling of an ample divisor over $Y$ which terminates, and we obtain a birational contraction $ \theta \colon X \dashrightarrow X' $. Set $\Delta':=\theta_*\Delta $ and let $\xi'\colon X'\to Y$ be the induced morphism. 
	
	By Lemma \ref{lem:trivMMP} there exists a small rational number $\delta$ such that, if we run a $ \big(K_{X'}+(1-\delta)\Delta' \big) $-MMP with scaling of an ample divisor over $Y$, then this MMP is $(K_{X'}+\Delta')$-trivial. Note that $ K_{X'}+(1-\delta)\Delta'$ is not $\xi'$-pseudoeffective: indeed, by possibly choosing $\delta$ smaller, we may assume that the map $\theta$ is $ \big(K_X+(1-\delta)\Delta\big)$-negative, and the claim follows since $K_X+(1-\delta)\Delta$ is not $\xi$-pseudoeffective by (b$_1$). Therefore, this relative $ \big(K_{X'}+(1-\delta)\Delta' \big) $-MMP terminates with a Mori fibre space $f''\colon X''\to Y''$ over $Y$ by \cite[Corollary 1.3.3]{BCHM}. Let $\theta' \colon X' \dashrightarrow X''$ denote that MMP and set $\Delta'':=\theta'_*\Delta'$. 
	
	\begin{center}
		\begin{tikzcd}
			X \arrow[r, dashed, "\theta"] \arrow[dr, "\xi" swap] & X' \arrow[r, dashed, "\theta'"] \arrow[d, "\xi'"] & X'' \arrow[d, "f''"] \arrow[dl] \\
			& Y & Y'' \arrow[l]
		\end{tikzcd}
	\end{center} 
	
	Then the pair $ (X'',\Delta'') $ is $ \Q $-factorial log canonical, the pair $(X'',0)$ is klt since the pair $(X'',(1-\delta)\Delta'')$ is dlt, and we have 
	$$K_{X''} + \Delta'' \equiv_{Y''} 0$$
	by Lemma \ref{lem:trivMMP}. Furthermore, the numerical equivalence over $Y''$ coincides with the $\R$-linear equivalence over $Y''$, since $f''$ is an extremal contraction \cite[Theorem 3.25(4)]{KM98}.
	
	If $K_{X''}+\Delta''$ has an NQC weak Zariski decomposition over $ Z $, then $ K_X + \Delta $ has an NQC weak Zariski decomposition over $ Z $ by Lemma \ref{lem:updownWZD}.
	
	Therefore, by replacing $(X,\Delta)$ by $(X'',\Delta'')$ and by setting $T:=Y''$ and $g:=f''$, we achieve Assumption 2.
	
	\medskip
	
	\emph{Step 4.}
	By \cite[Proposition 3.2(3)]{Bir11} there exist $\Q$-divisors $\Delta_1,\dots,\Delta_m$ such that each pair $(X,\Delta_i)$ is log canonical with $ K_X + \Delta_i $ nef over $T$, and there exist positive real numbers $r_1,\dots,r_m$ such that $\sum r_i=1$ and
	\begin{equation}\label{eq:2}
		K_X+\Delta=\sum\limits_{i=1}^m r_i(K_X+\Delta_i).
	\end{equation}
	Then clearly $K_X+\Delta_i\equiv_T0$ for all $i$ by (b$_2$) and by \eqref{eq:2}, so $K_X+\Delta_i\sim_{\Q,T}0$ for all $i$ by (c$_2$). Hence, by \cite[Theorem 3.6]{FG14} for each $ i \in \{1,\dots,m\}$ there exist log canonical g-pairs $(T/Z,B_i+M_i)$ on $T$ such that the $B_i$ and $M_i$ are $\Q$-divisors and
	\begin{equation}\label{eq:3}
	K_X+\Delta_i\sim_\Q g^*(K_T+B_i+M_i).
	\end{equation}
	By \eqref{eq:2} and \eqref{eq:3} we obtain 
	\begin{equation}\label{eq:5}
	K_X + \Delta = g^* (K_T + B_T + M_T),
	\end{equation} 
	where $ B_T := \sum r_i B_i $ and $ M_T := \sum r_i M_i $.
	By construction, the resul\-ting g-pair $ (T/Z, B_T + M_T) $ is NQC log canonical and $ K_T + B_T + M_T $ is pseudoeffective over $Z$ by \eqref{eq:5}. Hence, by assumptions of the theorem and by (a$_2$), the g-pair $ (T/Z, B_T + M_T) $ admits an NQC weak Zariski decomposition over $ Z $, which induces an NQC weak Zariski decomposition over $ Z $ for $ (X/Z,\Delta) $ due to \eqref{eq:5} and Remark \ref{rem:biratWZD}, as desired.
\end{proof}

\begin{thm}\label{thm:HM}
	Assume the existence of NQC weak Zariski decompositions for NQC log canonical g-pairs of dimensions at most $n-1$. 
	
	Let $(X/Z,\Delta+M)$ be a pseudoeffective NQC $\Q$-factorial dlt g-pair of dimension $n$ such that for each $\varepsilon>0$ the divisor $K_X+\Delta+(1-\varepsilon)M$ is not pseudoeffective over $Z$. Then $(X,\Delta+M)$ has an NQC weak Zariski decomposition over $Z$.
\end{thm}

\begin{proof}
	The proof follows the same strategy as that of \cite[Theorem 2]{HM18}, and we include all the details. 
		
	By Lemma \ref{lem:Gongyo} there exists a birational contraction $ \varphi\colon X\dashrightarrow S$ over $Z$ and a fibration $f\colon S\to Y$ over $Z$ such that, if we set $\Delta_S:=\varphi_*\Delta$ and $M_S:=\varphi_*M$, then:
	\begin{enumerate}
		\item[(a)] $ (S,\Delta_S+M_S) $ is a $\Q$-factorial log canonical g-pair, 
		\item[(b)] $ Y $ is a normal quasi-projective variety with $\dim Y<\dim X$,
		\item[(c)] $ K_S + \Delta_S +M_S\sim_{\R,Y} 0 $,
		\item[(d)] $\varphi$ is a $\big(K_X+\Delta+(1-\varepsilon)M\big)$-MMP for some $0<\varepsilon\ll1$ and $f$ is the corresponding Mori fibre space.
	\end{enumerate}
	As in Step 1 of the proof of Theorem \ref{thm:maintechnical}, by replacing $X$ by a higher model we may assume the following:
	
	\medskip
	
	\emph{Assumption 1.}
	There exists a fibration $\xi\colon X\to Y$ over $Z$ to a normal quasi-projective variety $Y$ such that $\dim Y<\dim X$ and such that:
	\begin{enumerate}
		\item[(a$_1$)] $\nu\big(F,(K_X+\Delta+M)|_F\big)=0$ and $h^1(F,\OO_F)=0$ for a very general fibre $F$ of $\xi$, 
		\item[(b$_1$)] $K_X+\Delta+(1-\varepsilon)M$ is not $\xi$-pseudoeffective for any $\varepsilon>0$. 
	\end{enumerate}
	
	\medskip
	
	If $\dim Y=0$ (and thus necessarily $\dim Z=0$), then $K_X+\Delta+M$ has an NQC weak Zariski decomposition as in Step 2 of the proof of Theorem \ref{thm:maintechnical}, and we are done. 
	
	\medskip
	
	If $\dim Y>0$, note that $(X,\Delta+M)$ is also a g-pair over $Y$. It follows by (a$_1$) and by Corollary \ref{cor:relativeMMPcontracting} that the divisor $K_X+\Delta+M$ is effective over $Y$. By Theorem \ref{thm:HanLi} we may run a $(K_X+\Delta+M)$-MMP with scaling of an ample divisor over $Y$ which terminates, and we obtain a birational contraction $ \theta \colon X \dashrightarrow X' $. Set $\Delta':=\theta_*\Delta $ and $M':=\theta_*M$, and let $\xi'\colon X'\to Y$ be the induced morphism. 
	
	By Lemma \ref{lem:trivMMPg} there exists a small rational number $\delta$ such that, if we run a $ \big(K_{X'}+\Delta'+(1-\delta)M' \big) $-MMP with scaling of an ample divisor $A$ over $Y$, then this MMP is $(K_{X'}+\Delta'+M')$-trivial. Note that this MMP is also a $(K_{X'}+\Delta'+(1-\delta)M'+\varepsilon A)$-MMP for some $0<\varepsilon\ll1$, hence it is the MMP with scaling of $A$ over $Y$ for some klt pair by \cite[Lemma 3.5]{HanLi}. 
	
	The divisor $ K_{X'}+\Delta'+(1-\delta)M'$ is not $\xi'$-pseudoeffective by (b$_1$) as in Step 3 of the proof of Theorem \ref{thm:maintechnical}. Therefore, this relative $ \big(K_{X'}+\Delta'+(1-\delta)M' \big) $-MMP terminates with a Mori fibre space $f''\colon X''\to Y''$ over $Y$ as in the proof of Lemma \ref{lem:Gongyo}. Let $\theta'\colon X'\dashrightarrow X''$ denote that MMP, and set $\Delta'':=\theta'_*\Delta'$ and $M'':=\theta'_*M'$. 
		
	Then the NQC g-pair $ (X'',\Delta''+M'') $ is $ \Q $-factorial log canonical, the pair $(X'',0)$ is klt by \cite[Remark 2.3]{HanLi} since the g-pair 
	$ (X'',\Delta''+(1-\delta)M'') $ is dlt, and we have 
	\begin{equation}\label{eq:78}	
	K_{X''} + \Delta''+M'' \equiv_{Y''} 0
	\end{equation}
	by Lemma \ref{lem:trivMMPg}. Furthermore, the numerical equivalence over $Y''$ coincides with the $\R$-linear equivalence over $Y''$ since $f''$ is an extremal contraction \cite[Theorem 3.25(4)]{KM98}. Moreover, the divisor $ M'' $ is ample over $ Y'' $ by \eqref{eq:78} and since ${-}(K_{X''}+\Delta''+(1-\delta)M'')$ is $f''$-ample.
	
	Then as in Step 3 of the proof of Theorem \ref{thm:maintechnical} we may replace $(X,\Delta+M)$ by $(X'',\Delta''+M'')$ and set $T:=Y''$, and thus we may assume the following:
	
	\medskip
	
	\emph{Assumption 2.}
	There exists a fibration $g\colon X\to T$ to a normal quasi-projective variety $T/Z$ such that:
	\begin{enumerate}
	\item[(a$_2$)] $0<\dim T<\dim X$, 
	\item[(b$_2$)] $K_X+\Delta+M\equiv_T 0$,
	\item[(c$_2$)] the numerical equivalence over $T$ coincides with the $\R$-linear equi\-valence over $T$,
	\item[(d$_2$)] $\rho(X/T)=1$,
	\item[(e$_2$)] $M$ is ample over $T$.
	\end{enumerate} 
\emph{However, instead of the g-pair $(X,\Delta+M)$ being $\Q$-factorial dlt, we may only assume that it is a $\Q$-factorial log canonical g-pair and $(X,0)$ is klt.}
	
	\medskip
	
	\emph{Step 4.}
	The g-pair $(X/Z,\Delta+M)$ is $\Q$-factorial log canonical and such that $(X,0)$ is klt and $K_X+\Delta+M$ is nef over $ T $. Since $\NEb(X/T)$ is extremal in $\NEb(X/Z)$ by \cite[p.\ 12]{Deb01}, by applying \cite[Proposition 3.16]{HanLi} to the collection of extremal rays of $\NEb(X/Z)$ corresponding to $\NEb(X/T)$, we deduce that there exist $\Q$-divisors $\Delta_1,\dots,\Delta_m$ and $ M_1,\dots,M_m $ such that each g-pair $(X/Z,\Delta_i + M_i)$ is log canonical with $ K_X + \Delta_i + M_i $ nef over $T$, and there exist positive real numbers $r_1,\dots,r_m$ such that $\sum r_i=1$ and
	\begin{equation}\label{eq:6}
	K_X+\Delta+M=\sum\limits_{i=1}^m r_i(K_X+\Delta_i+M_i).
	\end{equation}
	Then clearly $K_X+\Delta_i+M_i\equiv_T0$ for all $i$ by (b$_2$) and by \eqref{eq:6}, and thus $K_X+\Delta_i+M_i\sim_{\Q,T}0$ for all $ i $ by (c$_2$). Hence, by \cite[Theorem 6]{Fil19}\footnote{This is a generalisation of \cite[Theorem 1.4]{Fil18} to the context of projective morphisms of quasi-projective g-pairs with rational boundary divisors and rational nef divisors.} for each $ i \in \{1,\dots,m\}$ there exist log canonical g-pairs $(T/Z,B_i+N_i)$ on $T$ such that the $B_i$ and $N_i$ are $\Q$-divisors and
	\begin{equation}\label{eq:7}
	K_X+\Delta_i+M_i\sim_\Q g^*(K_T+B_i+N_i).
	\end{equation}
	By \eqref{eq:6} and \eqref{eq:7} we obtain
	\begin{equation}\label{eq:8}
	K_X + \Delta + M = g^* (K_T + B_T + N_T), 
	\end{equation}
	where $ B_T := \sum r_i B_i $ and $ N_T := \sum r_i N_i $. By construction, the resulting g-pair $ (T/Z, B_T + N_T) $ is NQC log canonical and $ K_T + B_T + N_T $ is pseudoeffective over $ Z $ by \eqref{eq:8}. Hence, by assumptions of the theorem and by (a$_2$), the g-pair $ (T/Z, B_T + N_T) $ admits an NQC weak Zariski decomposition over $Z$, which induces an NQC weak Zariski decomposition over $Z$ for $(X/Z,\Delta+M)$ due to \eqref{eq:8} and Remark \ref{rem:biratWZD}, as desired.
\end{proof}

\begin{rem}
The canonical bundle formula from \cite{Fil19} is extended to NQC g-pairs in \cite{HanLiu19}, which can be used to shorten Step 4 of the previous proof.
\end{rem}

\begin{cor}\label{cor:HM}
	Assume that log canonical pairs of dimensions at most $n$ have NQC weak Zariski decompositions. Then NQC log canonical g-pairs of dimension $n$ have NQC weak Zariski decompositions. 
\end{cor}

\begin{proof}
	The proof is by induction on the dimension $n$. We may therefore assume that NQC log canonical g-pairs of dimensions at most $n-1$ have NQC weak Zariski decompositions. 
	
	Let $(X/Z,\Delta+M)$ be a pseudoeffective NQC log canonical g-pair. By passing to a dlt blowup and by Remark \ref{rem:biratWZD} we may assume that it is a pseudoeffective NQC $ \Q $-factorial dlt g-pair.  Set
	$$\mu=\inf\big\{t\in\R_{\geq0}\mid K_X+\Delta+tM\text{ is pseudoeffective over }Z\big\}.$$
	We distinguish two cases.
	
	\medskip
	
	Assume first that $ \mu = 0 $. By Remark \ref{rem:q-cartier} the pseudoeffective pair $(X,\Delta)$ is log canonical, hence it has an NQC weak Zariski decomposition over $Z$ by assumption, and therefore so does the g-pair $(X,\Delta+M)$ by Remark \ref{rem:addM}.
	
	\medskip
	
	Assume now that $0<\mu\leq1$. It follows by Theorem \ref{thm:HM} that the g-pair $(X,\Delta+\mu M)$ has an NQC weak Zariski decomposition over $ Z $, and therefore so does the g-pair $(X,\Delta+M)$ by Remark \ref{rem:addM}.
\end{proof}

We can now give the proof of Theorem \ref{thm:HMref}.

\begin{proof}[Proof of Theorem \ref{thm:HMref}]
	By Lemma \ref{lem:downthedimension} we may assume the existence of NQC weak Zariski decompositions for smooth varieties of dimensions at most $n$. By induction on the dimension we may assume the existence of NQC weak Zariski decompositions for NQC log canonical g-pairs of dimensions at most $n-1$. Thus, by Corollary \ref{cor:HM} it suffices to show the existence of NQC weak Zariski decompositions for log canonical pairs of dimension $n$.
	
	Let $(X/Z,\Delta)$ be a pseudoeffective log canonical pair of dimension $n$. By passing to a dlt blowup and by Remark \ref{rem:biratWZD} we may assume that the pair $(X,\Delta)$ is $ \Q $-factorial dlt. Furthermore, by passing to a log resolution and by Lemma \ref{lem:updownWZD} we may assume that the pair $(X,\Delta)$ is log smooth. Set
	$$\tau:=\inf\{t\in \R_{\geq0}\mid K_X+t\Delta\text{ is pseudoeffective over }Z\}.$$
	We distinguish two cases.
	
	\medskip
	
	Assume first that $\tau=0$. Then the pseudoeffective smooth pair $(X,0)$ has an NQC weak Zariski decomposition over $Z$ by assumption, and therefore so does the pair $(X,\Delta)$. 
	
	\medskip
	
	Assume now that $0<\tau\leq1$. It follows by Theorem \ref{thm:maintechnical} that the pair $ (X,\tau \Delta) $ has an NQC weak Zariski decomposition over $Z$, and therefore so does the pair $(X,\Delta)$. 
\end{proof}

\section{Proofs of the main results}\label{sec:main1}

In this section we prove the remaining results announced in the introduction. 

\begin{proof}[Proof of Theorem \ref{thm:mainthm}]
By assumption, by Remark \ref{rem:MMimplWZD} and by Lemma \ref{lem:downthedimension} we may assume the existence of NQC weak Zariski decompositions for smooth varieties of dimensions at most $n$. 

Let $(X/Z,\Delta)$ be a pseudoeffective log canonical pair of dimension $n$. By Theorem \ref{thm:HMref} we may assume the existence of NQC weak Zariski decompositions for NQC log canonical g-pairs of dimensions at most $n-1$, and that the pair 
$ (X,\Delta) $ has an NQC weak Zariski decomposition over $Z$. We conclude by Theorem \ref{thm:HH_gMM}. 
\end{proof}

We have also the following version of Theorem \ref{thm:mainthm} for g-pairs:

\begin{thm} \label{thm:mainthm_g-pairs}
	Assume the existence of minimal models for smooth varieties of dimension $n$. 
	
	Let $ (X/Z,\Delta+M) $ be a pseudoeffective NQC log canonical g-pair of dimension $ n $. Then the following statements hold.
	\begin{enumerate}
		\item[(i)] The g-pair $ (X,\Delta+M) $ has a minimal model in the sense of Birkar-Shokurov over $Z$.
		\item[(ii)] If additionally $X$ is $ \Q $-factorial and $(X,0)$ is klt, then the g-pair $(X,\Delta +M)$ has a minimal model over $Z$.
	\end{enumerate}	
\end{thm}

\begin{proof}
	The proof is identical to the proof of Theorem \ref{thm:mainthm}, replacing Theorem \ref{thm:HH_gMM} by Theorem \ref{thm:HanLi}.
\end{proof}

\begin{proof}[Proof of Theorem \ref{thm:maincorollary}]
	Note that (i) follows by (ii) by Remark \ref{rem:MMimplWZD}. 
	
	Conversely, assume that the pair $(X,\Delta)$ has an NQC weak Zariski decomposition over $Z$. As in the proof of Theorem \ref{thm:mainthm} we may assume the existence of NQC weak Zariski decompositions for NQC log canonical g-pairs of dimensions at most $n-1$. We conclude by Theorem \ref{thm:HH_gMM}. 
\end{proof}

The same proof, replacing Theorem \ref{thm:HH_gMM} by Theorem \ref{thm:HanLi}, gives the following analogue for g-pairs.

\begin{thm}\label{thm:maincorollary_g-pairs}
	Assume the existence of minimal models for smooth varieties of dimension $n-1$. 
	
	Let $ (X/Z,\Delta+M) $ be a $\Q$-factorial NQC log canonical g-pair of dimension $ n $ such that $ (X,0) $ is klt. The following are equivalent:
	\begin{enumerate}[label=(\roman*)]
		\item $ (X,\Delta+M) $ has an NQC weak Zariski decomposition over $Z$,
		\item $ (X,\Delta+M) $ has a minimal model over $Z$. 
	\end{enumerate} 
\end{thm}

\begin{proof}[Proof of Theorem \ref{thm:uniruled}]
	As in the proof of Theorem \ref{thm:mainthm} we may assume the existence of NQC weak Zariski decompositions for NQC log canonical g-pairs of dimensions at most $n-1$. 
	
	Now, as in the second paragraph of the proof of Theorem \ref{thm:HMref} we may assume that the pair $(X,\Delta)$ is log smooth. Then $ K_X $ is not pseudoeffective over $Z$ by \cite[Corollary 0.3]{BDPP}. Set 
	$$\tau:=\inf\big\{t\in \R_{\geq0}\mid K_X+t\Delta\text{ is pseudoeffective over }Z\big\}.$$
	Then $0<\tau\leq1$. It follows by Theorem \ref{thm:maintechnical} that the pair $(X,\tau \Delta)$ has an NQC weak Zariski decomposition over $Z$, and therefore so does the pair $(X,\Delta)$. We conclude by Theorem \ref{thm:HH_gMM}.
\end{proof}

For g-pairs, we have:

\begin{thm} \label{thm:uniruled_g-pairs}
	Assume the existence of minimal models for smooth varieties of dimension $n-1$.
	
	Let $ (X/Z,\Delta+M) $ be a pseudoeffective NQC log canonical g-pair of dimension $ n $ such that a general fibre of the morphism $ X\to Z $ is uniruled. Then $ (X,\Delta+M) $ has a minimal model in the sense of Birkar-Shokurov over $ Z $. 
	
	If additionally $X$ is $ \Q $-factorial and $ (X,0) $ is klt, then $ (X,\Delta+M) $ has a minimal model over $ Z $. 
\end{thm}

\begin{proof}
	As in the beginning of the proof of Theorem \ref{thm:uniruled} we may assume the existence of NQC weak Zariski decompositions for NQC log canonical g-pairs of dimensions at most $n-1$, and that the g-pair $(X,\Delta+M)$ is log smooth. Set 
	$$\tau_M:=\inf\big\{t\in \R_{\geq0}\mid K_X+\Delta+tM\text{ is pseudoeffective over }Z\big\}.$$
	Then $0\leq\tau_M\leq1$. We claim now that the g-pair $ (X,\Delta+\tau_M M) $ has an NQC weak Zariski decomposition over $Z$. Indeed, if $\tau_M=0$, then $(X,\Delta)$ has an NQC weak Zariski decomposition over $Z$ by the proof of Theorem \ref{thm:uniruled}, while if $ 0 < \tau_M \leq 1 $, then the assertion follows by Theorem \ref{thm:HM}. Hence, the g-pair $(X,\Delta+M)$ has an NQC weak Zariski decomposition over $Z$ by Remark \ref{rem:addM}. We conclude by Theorem \ref{thm:HH_gMM}.
\end{proof}

\begin{proof}[Proof of Corollary \ref{cor:dim4+5}]
	The existence of minimal models for terminal $4$-folds over $ Z $ follows by \cite[Theorem 5-1-15]{KMM87}. Consequently, (i) follows by Theorem \ref{thm:mainthm}, (ii) follows by Theorem \ref{thm:maincorollary}, and (iii) follows by Theorem \ref{thm:uniruled}.
\end{proof} 

\begin{proof}[Proof of Theorem \ref{thm:scaling}]
	As in the proof of Theorem \ref{thm:mainthm} we may assume the existence of NQC weak Zariski decompositions for NQC log canonical g-pairs of dimensions at most $n-1$. Then the result follows by Theorem \ref{thm:HH_gMM}. 
\end{proof}

Similarly, by replacing Theorem \ref{thm:HH_gMM} by Theorem \ref{thm:HanLi} in the proof above we obtain:

\begin{thm}\label{thm:scalingGpairs}
	Assume the existence of minimal models for smooth varieties of dimension $n-1$.
	
	Let $ (X/Z,\Delta+M) $ be an NQC log canonical g-pair of dimension $ n $ which has an NQC weak Zariski decomposition over $Z$. 
\begin{enumerate}
\item[(i)] Then $ (X,\Delta+M) $ has a minimal model in the sense of Birkar-Shokurov over $Z$.
\item[(ii)] If additionally $X$ is $ \Q $-factorial and $(X,0)$ is klt, then any $(K_X + \Delta +M)$-MMP with scaling of an ample divisor over $Z$ terminates.
\end{enumerate}	
\end{thm}

\begin{proof}[Proof of Corollary \ref{cor:secondary}]
	By Theorem \ref{thm:HMref} we deduce the existence of NQC weak Zariski decompositions for NQC log canonical g-pairs of dimension $n$. Consequently, the result follows by \cite[Theorem 1]{HM18}; note that the results in op.\ cit.\ are stated for rational divisors, but their proof goes through verbatim for NQC g-pairs.
\end{proof}
	
	\bibliographystyle{amsalpha}
	\bibliography{biblio}

\end{document}